\documentclass[a4paper,
reqno,12pt]{amsart} 
    \usepackage{textcomp}
    \usepackage{euscript}
    \usepackage[pdftex]{hyperref}
 \usepackage{colortbl}
\usepackage[all]{xy}
 \newtheorem{thm}{Theorem}[section]    \newtheorem{lem}[thm]{Lemma}   
 \theoremstyle{definition}    \theoremstyle{remark}  \newtheorem{rem}[thm]{Remark}
    \numberwithin{equation}{section}
\usepackage[top=2cm, bottom=2cm, left=2cm, right=2cm]{geometry}
\newcommand{\norm}[1]{\left\Vert#1\right\Vert} \newcommand{\scal}[1]{\left<#1\right>}
\newcommand{\Hq}{\mathbb H} \newcommand{\Sq}{\mathbb S}  
  \newcommand{\R}{\mathbb{R}}      
  \newcommand{\C}{\mathbb{C}}

\newcommand{\Lnur}{L^{2,\nu}(\R,\C)}
\newcommand{\Lnurd}{L^{2,\nu}(\R^2,\C)}
\newcommand{\Lnuc}{L^{2,\nu}(\C,\C)}
\newcommand{\Lnucd}{L^{2,\nu}(\C^2,\C)}
\newcommand{\Lnurh}{L^{2,\nu}(\R,\Hq)}

      \newcommand{\bz}{\overline{z}} \newcommand{\bw}{\overline{w}}

\begin{document}

\title[Composition of Segal-Bargmann transforms]{Composition of Segal-Bargmann transforms} 

\author{A. Benahmadi} 
\author{K. Diki}   
\author{A. Ghanmi}      
\address{ 
           CeReMAR, A.G.S., L.A.M.A., Department of Mathematics, P.O. Box 1014, 
           Faculty of Sciences, Mohammed V University in Rabat, Morocco}
 \email{abdelhadi.benahmadi@gmail.com}
 \email{kamal.diki@gmail.com}
 \email{amalelhamyani@gmail.com}
 \email{ag@fsr.ac.ma}
\begin{abstract}
 We introduce and discuss some basic properties of some integral transforms in the framework of specific functional Hilbert spaces, the holomorphic
 Bargmann-Fock spaces on $\C$ and $\C^2$ and the sliced hyperholomorphic Bargmann-Fock space on $\Hq$.
 The first one is a natural integral transform mapping isometrically the standard Hilbert space on the real line into
 the two-dimensional Bargmann-Fock space. It is obtained as composition of the one and two dimensional Segal-Bargmann transforms and reduces further to an extremely integral operator that looks like a composition
 operator of the one-dimensional Segal-Bargmann transform with a specific symbol.
 We study its basic properties, including the identification of its image and the determination of a like-left inverse defined on the whole two-dimensional Bargmann-Fock space. We also examine their combination with the Fourier transform which lead to special integral transforms connecting
 the two-dimensional Bargmann-Fock space and its analogue on the complex plane.
 We also investigate the relationship between special subspaces of the two-dimensional Bargmann-Fock space and the slice-hyperholomorphic one on the quaternions by introducing appropriate integral transforms. We identify their image and their action on the reproducing kernel.
 \end{abstract}


\maketitle

\section{Introduction} \label{s1}

The standard Segal-Bargmann transform intertwines the Schr\"odinger representation and the complex wave representation of
the quantum mechanical harmonic oscillator and plays an important role in quantum optics, in signal processing and in harmonic analysis on phase space
\cite{Bargmann1961,Folland1989,Zhu2012,Hall2013,Neretin1972}.
Such transform has also found several applications in the theory of slice regular functions on quaternions and the slice monogenic functions on Clifford algebras \cite{DMNQ2016,KMNQ2016,PSS2016,MouraoNQ2017,DG1.2017,CD2017}.
Its action on $L^{2,\nu}(\R^d,\C)$, $\nu>0$, the Hilbert space of $\C$-valued $e^{-\nu x^2}dx$-square integrable functions on the real line, is given by (with a slightly different convention from the classical one)
  \begin{eqnarray}\label{dSBT}
\mathcal{B}^{d,\nu} f (z) := c^\nu_d \int_{\R^d} f(x) e^{-\nu  \left(x  -  \frac{z}{\sqrt{2}}\right)^2}   dx; \quad c^\nu_d:= \left(\frac{\nu}{\pi}\right)^{\frac{3d}{4}} ,
 \end{eqnarray}
and made the quantum mechanical configuration space $L^{2,\nu}(\R^d,\C)$ unitarily isomorphic to the Bargmann-Fock space
   $$\mathcal{F}^{2,\nu}(\C ^d)   = Hol(\C^d) \cap  L^{2,\nu}(\C^d,\C),$$
 consisting of all $L^2$-holomorphic functions with respect to the Gaussian measure $e^{-\nu |z|^2}d\lambda$ on the $d$-dimensional complex space $\C^d$, $d\lambda$ being the Lebesgue measure on $\C^d$.

A quaternionic counterpart of $\mathcal{F}^{2,\nu}(\C)$ is the slice hyperholomorphic Bargmann-Fock space introduced in \cite{AlpayColomboSabadini2014}
 as
  \begin{align}\label{SliceBargmann}
   \mathcal{F}^{2,\nu}_{slice}(\Hq) = \mathcal{SR}(\Hq) \cap L^{2,\nu}(\C_I,\Hq),
  \end{align}
where $\mathcal{SR}(\Hq)$ denotes the space of (left) slice regular $\Hq$-valued functions on the quaternions and $L^{2,\nu}(\C_I,\Hq)$ is the Hilbert space of $\Hq$-valued $L^2$ functions with respect to the Gaussian measure on an given slice $\C_I=\R+ \R I$. The corresponding Segal-Bargmann transform $\mathcal{B}_\Hq^\nu$ is considered in \cite{DG1.2017} and maps the Hilbert space $\Lnurh$ of $\Hq$-valued functions that are $e^{-\nu x^2}dx$-square integrable on the real line onto $ \mathcal{F}^{2,\nu}_{slice}(\Hq)$.
Its kernel function arises naturally as the unique extension of its holomorphic counterpart involved in \eqref{dSBT} to a slice regular function.
It can also be realized as the generating function of the rescaled real Hermite polynomials
\begin{eqnarray} \label{wrHn}
H_n^\nu(x) := (-1)^n e^{\nu x^2} \frac{d^n}{dx^n}\left(e^{-\nu x^2}\right).
\end{eqnarray}

In the present paper, we deal with the following special transform
 \begin{equation}\label{G}
 \mathcal{G}^\nu f(z,w) = \left(\frac{\nu}{\pi}\right)^{\frac{1}{2}} \mathcal{C}_{\psi_1} (\mathcal{B}^{1,\nu} f)(z,w)
 \end{equation}
  obtained as the composition operator $\mathcal{C}_{\psi_1} f = f\circ {\psi_1}$ of the one-dimensional Segal-Bargmann transform $\mathcal{B}^{1,\nu}$ with the specific symbol $ {\psi_1}(z,w) = \frac{z+iw}{\sqrt{2}}.$
It is a special one-to-one transform mapping the standard Hilbert space $\Lnur$ on the real line into the two-dimensional Bargmann-Fock space $\mathcal{F}^{2,\nu}(\C^2)$ on the two-dimensional complex space.
We study its basic properties and characterize its image. Namely, we show that $\mathcal{G}^\nu (\Lnur)$ coincides with (Theorem \ref{MainThm2})
\begin{equation}\label{Image}
  \mathcal{A}^{2,\nu}(\C^2) := \left\{F\in \mathcal{F}^{2,\nu}(\C^2); \, \left(\frac{\partial}{\partial z} + i \frac{\partial}{\partial w}\right) F =0 \right\}.
\end{equation}
This was possible by realizing this transform in a natural way as the composition of the $1$d and $2$d Segal-Bargmann transforms (Theorem \ref{MainThm}),
\begin{equation}\label{Gc}
\mathcal{G}^\nu = \mathcal{B}^{2,\nu}\circ \mathcal{B}^{1,\nu}.
\end{equation}
Moreover, if $Proj$ denotes the orthogonal projection on the one-dimensional Bargmann-Fock space, we show that the transform
\begin{equation}\label{R}
\mathcal{R}^\nu:=(\mathcal{B}^{1,\nu})^{-1}\circ Proj\circ (\mathcal{B}^{2,\nu})^{-1}
\end{equation}
defined on the whole $\mathcal{F}^{2,\nu}(\C^2)$ is a like-left inverse of $\mathcal{G}^\nu$ that can be expressed in terms of the inverse of $\mathcal{B}^{1,\nu}$  and a composition operator with a specific symbol ${\psi_2}:\C \longrightarrow\C ^2$.
More explicitly, we have
\begin{eqnarray} \label{inverse}
\mathcal{R}^{\nu}F(x)= \left(\frac{\pi}{\nu}\right)^{\frac{1}{4}} \int_{\C} F\left(\frac{\xi}{\sqrt{2}},-i\frac{\xi}{\sqrt{2}}\right) e^{-\frac{\nu}{2}\overline{\xi}^2 +\sqrt{2}\nu x\overline{\xi}} e^{-\nu|\xi|^2}d\lambda(\xi) .
\end{eqnarray}
Further properties of the transform $\mathcal{G}^\nu$ when combined with the a rescaled Fourier transform are also investigated (Theorem \ref{thmLinverse}).
They give rise to two extremely integral operators connecting isometrically the one-dimensional Bargmann-Fock space $\mathcal{F}^{2,\nu}(\C)$ to the two-dimensional one $\mathcal{F}^{2,\nu}(\C^2)$.

The like-left inverse $\mathcal{G}^\nu$ in \eqref{inverse} as well as the quaternionic Segal-Bargmann transform $\mathcal{B}_\Hq^\nu$ are then employed to introduce and study the integral transform
\begin{equation} \label{D}
\mathcal{I}^\nu:=\mathcal{B}_\Hq^\nu \circ \mathcal{R}^\nu.
\end{equation}
It is defined on the two-dimensional Bargmann-Fock space $\mathcal{F}^{2,\nu}(\C^2)$ with range in the slice hyperholomorphic Bargmann-Fock space $\mathcal{F}^{2,\nu}_{slice}(\Hq)$ in \eqref{SliceBargmann}.
We show that $\mathcal{I}^\nu$ reduces further to the integral operator
  \begin{equation}\label{BR1}
 \mathcal{I}^\nu f  (q) =  \left(\frac{\nu}{\pi}\right) \int_{\C}
 f\left(\frac{\xi}{\sqrt{2}}, \frac{-i\xi}{\sqrt{2}} \right) K^\nu_{\Hq}(q,\xi)  e^{-\nu |\xi|^2} d\lambda(\xi).
 \end{equation}
 where $K^\nu_{\Hq}(q,\xi)$ is the reproducing kernel of $\mathcal{F}_{slice}^{2,\nu}(\Hq)$ (see Theorem \ref{Iintclosed}). The image $\mathcal{I}^\nu(\mathcal{F}^{2,\nu}(\C^2))$ is identified to be $\mathcal{F}_{slice,i}^{2,\nu}(\Hq)$ the space of slice (left) regular functions on the quaternions leaving invariant the slice $\C_i\simeq \C$ (see Theorem \ref{thmFourier}). Added to $\mathcal{I}^\nu$, we consider the integral transform
 $$\mathcal{J}^{\nu}:=\mathcal{G}^{\nu} \circ (\mathcal{B}^\nu_\Hq)^{-1}$$
from $\mathcal{F}_{slice,i}^{2,\nu}(\Hq)$ into $\mathcal{F}^{2,\nu}(\C^2)$ with image coinciding with $ \mathcal{A}^{2,\nu}(\C^2)$ (see Theorem \ref{I}).
The action of $\mathcal{I}^\nu$ and $\mathcal{J}^\nu$ on the bases and the reproducing kernels are given.
It turns out that these transforms connect the standard basis and the reproducing kernels of these two spaces.

   To present these ideas, we adopt the following structure:
  we study in Section 2 some basic properties of the integral transform \eqref{G}, we identify its image and the expression of its left inverse.
   Section 3 is devoted to describe the transform $\mathcal{I}^\nu$ given through \eqref{D} as well as its inverse, and to investigate the relationship between the classical Bargmann-Fock space on $\C^2$ and the slice-hyperholomorphic one on the quaternions.
    The appendix deals with a discussion concerning the high $2^n$-dimension as well as with some special integral transforms that are obtained as composition of $\mathcal{G}^\nu$ with the Fourier transform.

 \section{On composition of Segal-Bargmann transforms}
 The kernel function of the $d$-dimensional Segal-Bargmann transform $\mathcal{B}^{d,\nu}$ in \eqref{dSBT} is the analytic continuation to $\C^d$ of the standard Gaussian density on $\R^d$. It is given by
\begin{align}\label{KernelBTn}
A^{\nu}_d(z,x) = c^\nu_d e^{-\nu  \left(x  -  \frac{z}{\sqrt{2}}\right)^2}
\end{align}
with $z^2:=z_1^2+z_2^2+\cdots +z_d^2$ for $z=(z_1,\cdots ,z_d)\in{\C ^d}$.
Then, the integral transform in \eqref{G} acts on $\Lnur$ by
\begin{equation}\label{Ga}
 \mathcal{G}^\nu f(z,w)  := \left(\frac{\nu}{\pi}\right)^{\frac{1}{2}}   \int_\R f(x) A^{\nu}_1\left(\frac{z+iw}{\sqrt{2}},x\right)dx.
 \end{equation}
The following result shows that the transform $\mathcal{G}^\nu$ can be realized in a natural way by means of the Segal-Bargmann transforms $\mathcal{B}^{d,\nu}$; $d=1,2$,
according to the following diagram
      $$\xymatrix{
                  \Lnur        \ar[r]^{\mathcal{B}^{1,\nu}}
    \ar[rd]_{\mathcal{G}^\nu}                                        & \mathcal{F}^{2,\nu}(\C)  \ar[d]^{\mathcal{B}^{2,\nu}}\\
                                                                    & \mathcal{F}^{2,\nu}(\C^2)
  }$$

     \begin{thm}\label{MainThm} The above diagram is commutative, in the sense that we have $\mathcal{G}^\nu=\mathcal{B}^{2,\nu}\circ \mathcal{B}^{1,\nu}$ on $\Lnur$. Moreover, $\mathcal{G}^\nu$ defines an isometric operator mapping the Hilbert space $\Lnur$ into the Bargmann-Fock space $\mathcal{F}^{2,\nu}(\C ^2)$.
   \end{thm}

 \begin{proof}
For every given $\varphi\in \Lnur$, the function $\mathcal{B}^{2,\nu}\circ \mathcal{B}^{1,\nu}(\varphi)$ is clearly a holomorphic function on $\C^2$ and belongs to $\Lnucd$. Moreover, $\mathcal{B}^{2,\nu}\circ \mathcal{B}^{1,\nu}$ defines an isometric operator from $\Lnur$ into $\mathcal{F}^{2,\nu}(\C^2)$ since $\mathcal{B}^{2,\nu}$ and $ \mathcal{B}^{1,\nu}$ are.
To conclude for the proof of Theorem \ref{MainThm}, we only need to show that the diagram is commutative.
Thus, for every given $z,w\in{\C }$ and $x,y \in\R$, we have
   \begin{align}
   \mathcal{B}^{2,\nu}\circ \mathcal{B}^{1,\nu} f(z,w)
   &=   \int_{\R^2} \int_{\R} f(t) A^{\nu}_2((z,w),(x,y)) A^{\nu}_1(x+iy,t)  dtdxdy \nonumber
   \\&=    c^\nu_2 c^\nu_1 \int_{\R^2} \int_{\R}
   f(t) e^{-\nu  \left\{ \left(x - \frac{z}{\sqrt{2}}\right)^2 + \left(y - \frac{w}{\sqrt{2}}\right)^2 +
   \left(t - \frac{x+iy}{\sqrt{2}}\right)^2   \right\}} dtdxdy \nonumber
   \\ & \stackrel{(*)}{=}
  \left(\frac{\pi}{\nu}\right) c^\nu_2   \int_{\R} f(t) A^{\nu}_1 \left(\frac{z+iw}{\sqrt{2}},t\right)  dt. \label{diag1}
     \end{align}
The transition $(*)$ follows by direct computation, making appeal of the Fubini theorem as well as the explicit formula for the Gaussian integral.
The proof of the theorem is completed by comparing the right-hand side of \eqref{diag1} to \eqref{Ga}.
  \end{proof}

The next result identifies the image of $\Lnur$ by the one-to-one transform
$\mathcal{G}^\nu$, and characterizes it as the kernel $\ker_{\mathcal{F}^{2,\nu}(\C^2)}  (D_{z,w}) $ of the first order differential operator
$$ D_{z,w} := \frac{\partial}{\partial z} + i \frac{\partial}{\partial w} $$
acting on $\mathcal{F}^{2,\nu}(\C^2)$. More precisely, we assert the following

     \begin{thm}\label{MainThm2} Keep notations as above and define $\mathcal{A}^{2,\nu}(\C^2)$ as in \eqref{Image},
     $\mathcal{A}^{2,\nu}(\C^2) :=\ker_{\mathcal{F}^{2,\nu}(\C^2)}  (D_{z,w}) .$
Then, we have
\begin{enumerate}
  \item[(i)] $\mathcal{A}^{2,\nu}(\C^2)$ is a closed subspace of  $\mathcal{F}^{2,\nu}(\C^2)$.
  \item[(ii)] The functions $e_m^\nu(z,w):=  (z+iw)^m $  form an orthogonal basis of the Hilbert space $\mathcal{A}^{2,\nu}(\C^2)$.
  \item[(iii)] $\mathcal{A}^{2,\nu}(\C^2) = \mathcal{G}^\nu (\Lnur)$.
\end{enumerate}
   \end{thm}

     \begin{proof} Notice first that by the definition in \eqref{G} and the fact that
     $$\mathcal{B}^{1,\nu} (H_m^\nu)(\xi) = \left(\frac{\nu}{\pi}\right)^{{1}/{4}}\sqrt{2}^m\nu^m \xi^m ,$$ the action of $\mathcal{G}^\nu$ on the rescaled Hermite polynomials in \eqref{wrHn} is given by
     \begin{eqnarray}
     \mathcal{G}^\nu (H_m^\nu )(z,w) &=&  \left(\frac{\nu}{\pi}\right)^{\frac{1}{2}} \mathcal{B}^{1,\nu} (H_m^\nu)\left(\frac{z+iw}{\sqrt{2}}\right)
     \nonumber \\ &=&
     c^\nu_1\nu^m (z+iw)^m   . \label{basisImage}
     \end{eqnarray}
Subsequently, the image $\mathcal{G}^\nu (\Lnur)$ is then spanned by the functions $e_m^\nu(z,w):=  (z+iw)^m $, since the polynomials $H_k^\nu$ form an orthogonal basis of $\Lnur$. 
Accordingly, the proof of $(i)$ readily follows and then $\mathcal{A}^{2,\nu}(\C^2)$ is a Hilbert space for the scalar product induced from
$\mathcal{F}^{2,\nu}(\C^2)$, while $(iii)$ is an immediate consequence of $(ii)$.
Moreover, the functions $e_k^\nu(z,w)$ satisfy $D_{z,w} e_k^\nu(z,w)=0$ and form an orthogonal system in the Hilbert space $\mathcal{A}^{2,\nu}(\C^2)$.
To conclude for $(ii)$, we should prove completeness of $e_k^\nu(z,w)$ in $\mathcal{A}^{2,\nu}(\C^2)$.
To this end, let $F\in \mathcal{F}^{2,\nu}(\C^2)$ such that $D_{z,w} F=0$ and $\scal{F,e_k^\nu}_{\Lnucd} =0$ for all $k$ and show that $F$ is then identically zero on $\C^2$. Indeed, by expanding $F$ as series
 $F(z,w) =\sum\limits_{m,n=0}^{+\infty} a_{m,n} z^mw^n\in \mathcal{F}^{2,\nu}(\C^2)$, we show that
$$\scal{F,e_k^\nu}_{\Lnucd} = \sum_{j=0}^{k} \binom{k}{j}  (-i)^j a_{k-j,j}\norm{e_{k-j}}_{\Lnuc}^2\norm{e_{j}}_{\Lnuc}^2,$$
where $e_{j}(\xi)= \xi^j$.
Hence $\scal{F,e_k^\nu}_{\Lnucd}=0$, for every $k=0,1,\cdots $, implies that
\begin{equation}\label{Df0}
 \left(\frac{\pi}{\nu}\right)^2 \frac{k!}{\nu^k} \sum_{j=0}^{k}  \left(-i\right)^j a_{k-j,j} =0.
 \end{equation}
 Moreover, we can show that the condition $D_{z,w} F=0$ is equivalent to that
$$ a_{m+1,n} = - i \left(\frac{n+1}{m+1}\right)a_{m,n+1}$$
for all $m,n=0,1,\cdots $, which by induction infers
\begin{equation}\label{Recamn}
 a_{m,n} = i^n \left(\frac{(m+n)!}{m!n!}\right) a_{m+n,0} , \quad m= 0,1,\cdots; \, n=1,\cdots.
 \end{equation}
 Inserting this in \eqref{Df0}, it yields $a_{k,0}=0$ for all $k$ and therefore $a_{m,n}=0$ for all $m,n$ by means of \eqref{Recamn}.
 This yields the required result. 
\end{proof}

 \begin{rem} The space $\mathcal{A}^{2,\nu}(\C^2)$ is on interest in itself. It is the phase space in $2$-complex dimesion that is unitary isomorphic of the configuration space $\Lnur$. Moreover, the transform  $\mathcal{G}^\nu$ is a coherent state transform from $\Lnur$ onto $\mathcal{A}^{2,\nu}(\C^2)$ in \eqref{Image}, in the sense that
   its kernel function can be recovered as a bilinear generating function of the orthonormal bases of the Hilbert spaces $\Lnur$ and $\mathcal{A}^{2,\nu}(\C^2)$.
  \end{rem}

 \begin{rem}  The assertion $(iii)$ in Theorem \ref{MainThm2} shows that $ \mathcal{B}^{2,\nu}  (\mathcal{F}^{2,\nu}(\C))=\mathcal{A}^{2,\nu}(\C^2)$.
This is in fact contained in \eqref{basisImage}. Indeed, for $e_m(\xi)=\xi^m$, we have
  $$  \mathcal{B}^{2,\nu}\left(e_m \right)(z,w) = \left(\frac{\nu}{\pi}\right)^{\frac{1}{2}} \left(\frac 12 \right)^{\frac{m}{2}} e_m^\nu(z,w).$$
     \end{rem}

 \begin{rem}  The inverse transform of $\mathcal{G}^\nu$ is defined from $\mathcal{A}^{2,\nu}(\C^2)$ onto $\Lnur$ and is clearly given by
 $ (\mathcal{B}^{1,\nu})^{-1}\circ(\mathcal{B}^{2,\nu})^{-1}$ and coincides with the restriction to  $\mathcal{A}^{2,\nu}(\C^2)$ of $\mathcal{R}^\nu$ introduced below.    \end{rem}

Now, let us consider the transform $\mathcal{R}^\nu$ from $\mathcal{F}^{2,\nu}(\C ^2)$ into $\Lnur$ defined by the following commutative diagram $$\xymatrix{
    \mathcal{F}^{2,\nu}(\C^2) \ar[r]^{\mathcal{R}^\nu} \ar[d]_{(\mathcal{B}^{2,\nu})^{-1}} & \Lnur  \\
     \Lnurd  \ar[r]_{Proj} & \mathcal{F}^{2,\nu}(\C) \ar[u]_{(\mathcal{B}^{1,\nu})^{-1}},
  }$$
  where $Proj$ stands for the orthogonal projection from $\Lnuc$ onto the standard Bargmann-Fock space $\mathcal{F}^{2,\nu}(\C)$ and given by (see for example \cite{Zhu2012})
  \begin{align} \label{proj}
Proj f (\xi) = \left(\frac{\nu}{\pi}\right) \int_{\C} f(\zeta) e^{\nu\xi\overline{\zeta}} e^{-\nu|\zeta|^2}d\lambda(\zeta).
\end{align}
The following result gives an integral representation of the operator $\mathcal{R}^\nu:=(\mathcal{B}^{1,\nu})^{-1}\circ Proj\circ  (\mathcal{B}^{2,\nu})^{-1}$. It involves of the inverse of $\mathcal{B}^{1,\nu}$ and the composition operator
$\mathcal{C}_{{\psi_2}}F = F \circ {\psi_2}$ with the symbol function ${\psi_2}:\C\longrightarrow\C^2$  given by
\begin{align} \label{chi}
   {\psi_2}(\xi):= \left(\frac{\xi}{\sqrt{2}},-i\frac{\xi}{\sqrt{2}}\right).
   \end{align}

\begin{thm}\label{thmLinverse}
The transform $\mathcal{R}^\nu$ defined on the whole $\mathcal{F}^{2,\nu}(\C^2)$ looks like a left inverse of $\mathcal{G}^\nu$. Moreover, we have
  \begin{eqnarray} \label{inversechi}
  \mathcal{R}^\nu F=\displaystyle\left(\frac{\pi}{\nu}\right)^{\frac{1}{2}}(\mathcal{B}^{1,\nu})^{-1}(\mathcal{C}_{{\psi_2}}F)
  \end{eqnarray}
  for every $F\in \mathcal{F}^{2,\nu}(\C ^2) $ which explicitly reads,
  \begin{eqnarray} \label{inversea}
\mathcal{R}^{\nu}F(x)=\displaystyle\left(\frac{\pi}{\nu}\right)^{\frac{1}{4}}\int_\C  F\left(\frac{\xi}{\sqrt{2}},-i\frac{\xi}{\sqrt{2}}\right) e^{-\frac{\nu}{2}\overline{\xi}^2 +\sqrt{2}\nu x\overline{\xi}} e^{-\nu\vert{\xi}\vert^2}d\lambda(\xi) .
\end{eqnarray}
 \end{thm}

  \begin{proof}
  For every $f\in \Lnur$, the function $\mathcal{B}^{1,\nu} f$ belongs to the one-dimensional Bargmann-Fock space $\mathcal{F}^{2,\nu}(\C)$ and therefore
  $Proj(\mathcal{B}^{1,\nu}f)=\mathcal{B}^{1,\nu} f$, so that  $\mathcal{R}^\nu \circ \mathcal{G}^\nu=id_{\Lnur}$.
  This shows that $\mathcal{R}^\nu$ is a left inverse of $\mathcal{G}^\nu$.
 Moreover, making use of the integral representation of the orthogonal projection \eqref{proj} and of
 $$
 (\mathcal{B}^{2,\nu})^{-1} F (\zeta)
 = c^\nu_2 
 \int_{\C^2} F(z,w) e^{-\frac{\nu}{2}(\bz^2+\bw^2)+ \frac{\nu}{\sqrt{2}} (\zeta[\bz - i\bw] + \overline{\zeta}[\bz + i\bw]) } e^{-\nu(|z|^2+|w|^2)}d\lambda(z,w)
 $$
for $\zeta\in \C$ and $F\in \mathcal{F}^{2,\nu}(\C^2)$, we get
\begin{align*}
Proj   (\mathcal{B}^{2,\nu})^{-1} F(\xi)
=c^\nu_2 \int_{\C^2}e^{-\frac{\nu}{2}(\bz^2+\bw^2)} F(z,w)I(\xi,\bz,\bw) e^{-\nu(|z|^2+|w|^2)}d\lambda(z,w),
\end{align*}
 where for $\xi,z,w \in \C$ we have
  \begin{align*}
  I(\xi,\bz,\bw) &:=\left(\frac{\nu}{\pi}\right) \int_{\C} e^{-\nu |\zeta|^2 + \frac{\nu}{\sqrt{2}} (\zeta[\bz - i\bw] + \overline{\zeta}[\bz + i\bw + \sqrt{2}\xi]) } d\lambda(\zeta)
  \\& =  e^{\frac{\nu}{2}(\bz^2+\bw^2)+\nu\xi\frac{(\bz-i\bw)}{\sqrt{2}}}.
  \end{align*}
 Therefore, by the reproducing property for the two-dimensional Bargmann-Fock space $\mathcal{F}^{2,\nu}(\C^2)$, we obtain
  \begin{align*}
 Proj   (\mathcal{B}^{2,\nu})^{-1} F (\xi) &= c^\nu_2 
 \int_{\C^2} F(z,w) e^{\nu\left(\frac{\xi}{\sqrt{2}}\bz-\frac{i\xi}{\sqrt{2}}\bw\right)} e^{-\nu(|z|^2+|w|^2)}d\lambda(z,w) \label{kernelC2}
  \\&
  = \left(\frac{\pi}{\nu}\right)^{\frac{1}{2}} F\left(\frac{\xi}{\sqrt{2}},-i\frac{\xi}{\sqrt{2}}\right) \nonumber
  \end{align*}
  for
  \begin{equation}\label{Rpkernel2}
  K^\nu_2 \left((u,v),(z,w)\right) = \left(\frac{\nu}{\pi}\right)^2 e^{\nu\left(u\bz+v\bw\right)}
   \end{equation}
   being the reproducing kernel of $\mathcal{F}^{2,\nu}(\C^2)$.
\end{proof}

\section{Connecting holomorphic and slice hyperholomorphic Bargmann-Fock spaces}

The slice hyperholomorphic quaternionic Bargmann-Fock space $\mathcal{F}_{slice}^{2,\nu}(\Hq)$, considered in \cite{AlpayColomboSabadini2014}, is a quaternionic counterpart of the holomorphic Bargmann-Fock space $\mathcal{F}^{2,\nu}(\C)$.
It is defined to be the right $\Hq$-vector space of all slice left regular functions on $\Hq$, $F\in \mathcal{SR}(\Hq)$, subject to the norm boundedness
$\norm{F}_{\mathcal{F}_{slice}^{2,\nu}(\Hq)}^2 <  +\infty $. This norm is associated to the inner product
\begin{equation}\label{spfg}
\scal{F,G}_{\mathcal{F}_{slice}^{2,\nu}(\Hq)} = \int_{\C_I}\overline{G_I(q)}F_I(q)e^{-\nu |q|^2} d\lambda_I(q),
\end{equation}
where for given $I\in \Sq=\{I\in \Hq; \, I^2=-1\}$, the function $F_I = F|_{\C_I}$ denotes the restriction of $F$ to the slice $\C_I:=\R+I\R$ and $d\lambda_I(q)=dxdy$ for $q=x+yI$.
It was shown in \cite{AlpayColomboSabadini2014} that $\mathcal{F}_{slice}^{2,\nu}(\Hq)$ does not depend on the choice of the imaginary unit $I$ and is a reproducing kernel Hilbert space, whose the reproducing kernel is given by
\begin{equation}\label{RpKernelH}
K^\nu_{\Hq}(q,p)= \left(\frac{\nu}{\pi}\right) e_{*}^{\nu[q,\overline{p}]} := \left(\frac{\nu}{\pi}\right) \sum_{m=0}^{+\infty}
 \frac{\nu^m q^m \overline{p}^m}{m!} ; \quad p,q\in \Hq.
\end{equation}
This space is closely connected to $\Lnurh$, the Hilbert space of all $\Hq$-valued and $L^2$ functions on the real line with respect to the Gaussian measure. In fact, $\mathcal{F}_{slice}^{2,\nu}(\Hq)$ can be realized as the image of $\Lnurh$ by considering the quaterenionic Segal-Bargmann transform \cite{DG1.2017}
\begin{eqnarray} \label{defQSBT}
 \mathcal{B}_\Hq^\nu f(q) := c^\nu_1 \int_{\R} f(x) e^{-\nu  \left(x  -  \frac{q}{\sqrt{2}}\right)^2}   dx.
 \end{eqnarray}

 Its inverse transform mapping $\mathcal{F}_{slice}^{2,\nu}(\Hq)$ onto $\Lnurh$ is given by
 \begin{align}\label{inverseExpIntRep}
(\mathcal{B}_\Hq^\nu)^{-1}F(x)  &=  c^\nu_1 \int_{\C_I}  F_I(q) e^{-\frac{\nu}{2}\overline{q}^2 +\sqrt{2}\nu x\overline{q}} e^{-\nu|q|^2} d\lambda_I(q).
\end{align}
 Examples of slice hyperholomorphic functions in $\mathcal{F}_{slice}^{2,\nu}(\Hq)$ can also be obtained from the one of the standard Bargmann-Fock space on $\C$ by the extension lemma below.

\begin{lem}[\cite{ColomboSabadiniStruppa2011,GentiliStoppatoStruppa2013}]\label{extensionLem}
Let $\Omega_I=\Omega\cap \C_I$; $I\in \mathbb{S}$, be a symmetric domain in $\C_I$ with respect to the real axis such that $\Omega_I\cap \R$ is not empty  and $\overset{\sim}\Omega=\underset{x+yJ\in{\Omega}}\cup x+y\mathbb{S}$ be the symmetric completion of $\Omega_I$.
For every holomorphic function $F:\Omega_I\longrightarrow \Hq$, the function $Ext(F)$ defined by
$$Ext(F)(x+yJ):= \dfrac{1}{2}[F(x+yI)+F(x-yI)]+\frac{JI}{2}[F(x-yI)-F(x+yI)];  \quad J\in \mathbb{S},$$
extends $F$ to a regular function on $\overset{\sim}\Omega$. Moreover, $Ext(F)$ is the unique slice regular extension of $F$.
\end{lem}
This lemma can be extended to the context of the two-dimensional Bargmann-Fock space $\mathcal{F}^{2,\nu}(\C^2)$ on $\C^2$.
This lies on the simple idea that consists of considering an appropriate restriction operator from $\mathcal{F}^{2,\nu}(\C ^2)$ into $\mathcal{F}^{2,\nu}(\C )$ and next apply the extension Lemma \ref{extensionLem}. For example, one can consider
 $$ \mathcal{I}^\nu: F \longmapsto F\circ {\psi_2}  \longmapsto  Ext(F\circ {\psi_2})$$
from $\mathcal{F}^{2,\nu}(\C^2)$ into a specific subspace of $\mathcal{F}_{slice}^{2,\nu}(\Hq)$,  where ${\psi_2}:\C \longrightarrow\C ^2$ is the one defined in \eqref{chi}.
The following result shows that the transform $\mathcal{I}^\nu$ is in fact realized by the following commutative diagram
\[
\xymatrix{
\mathcal{F}^{2,\nu}(\C^2)    \ar[r]^{\mathcal{I}^\nu} \ar[d]_{\mathcal{R}^\nu}     & \mathcal{F}_{slice}^{2,\nu}(\Hq)\\
\Lnur             \ar[r]_{inj}                         & \Lnurh \ar[u]_{\mathcal{B}^\nu_\Hq }}
\]
where $\mathcal{B}^\nu_\Hq$ is the quaternionic Segal-Bargmann transform in \eqref{defQSBT} and $\mathcal{R}^\nu$ is the transform given by \eqref{inversechi}.

\begin{thm} \label{Iintclosed}
The transform $\mathcal{B}^\nu_\Hq \circ \mathcal{R}^\nu$ coincides with $\mathcal{I}^\nu$ and acts on $\mathcal{F}^{2,\nu}(\C^2)$ by
     \begin{equation}\label{BR1}
 \mathcal{B}^\nu_\Hq \circ \mathcal{R}^\nu F  (q) =  \left(\frac{\nu}{\pi}\right) \int_{\C}
 F\left(\frac{\xi}{\sqrt{2}}, \frac{-i\xi}{\sqrt{2}} \right) K^\nu_{\Hq}(q,\xi)  e^{-\nu |\xi|^2} d\lambda(\xi),
 \end{equation}
 where $K^\nu_{\Hq}(q,\xi)$ is the reproducing kernel of $\mathcal{F}_{slice}^{2,\nu}(\Hq)$ as given by \eqref{RpKernelH}.
\end{thm}

For the proof, we will make use of the identity principle for slice regular functions
 \begin{lem}[\cite{ColomboSabadiniStruppa2011,GentiliStoppatoStruppa2013}]\label{IdentityPrinciple}
Let $F$ be a slice regular function on a slice domain $\Omega$ and denote by $\mathcal{Z}_F$ its zero set.
If $\mathcal{Z}_F \cap \C_I$ has an accumulation point in $\Omega_I$ for some $I\in \Sq$, then $F$ vanishes identically on $\Omega$.
\end{lem}

  \begin{proof}
  On one hand, the function $ \mathcal{B}^\nu_\Hq \circ \mathcal{R}^\nu F$ is slice regular by construction.
  On the other hand, one can show easily that the function $\xi \longmapsto F\left(\frac{\xi}{\sqrt{2}}, \frac{-i\xi}{\sqrt{2}} \right) $ belongs to the one-dimensional Bargmann-Fock space $\mathcal{F}^{2,\nu}(\C)$ for every $F\in \mathcal{F}^{2,\nu}(\C^2)$, and therefore its extension
   given by Lemma \ref{extensionLem}, is slice regular and belongs to $\mathcal{F}_{slice}^{2,\nu}(\Hq)$.
  Moreover, by means of the reproducing property for the elements in $\mathcal{F}_{slice}^{2,\nu}(\Hq)$, we obtain the following identity
       \begin{equation}\label{BR1}
 \mathcal{I}^\nu F  (q) =  \left(\frac{\nu}{\pi}\right) \int_{\C}
 F\left(\frac{\xi}{\sqrt{2}}, \frac{-i\xi}{\sqrt{2}} \right) K^\nu_{\Hq}(q,\xi)  e^{-\nu |\xi|^2} d\lambda(\xi).
 \end{equation}
To conclude that $\mathcal{B}^\nu_\Hq \circ \mathcal{R}^\nu F$ and $\mathcal{I}^\nu F$ are identically the same, we need only to prove it for their restrictions on $\C_i\simeq \C$ and then apply the identity principle for the slice left regular functions (Lemma \ref{IdentityPrinciple}).
To this end, we begin by rewriting the transforms $\mathcal{B}^\nu_\Hq$ and $\mathcal{R}^\nu$ as
  \begin{align*}
 \mathcal{B}^\nu_\Hq  f(q) &= \scal{f, \overline{S^\nu(q,\cdot)}}_{\Lnur} 
 = \int_{\R} f(x) S^\nu(q,x) e^{-\nu x^2} dx
 \end{align*}
 and
    \begin{align*}
 \mathcal{R}^\nu F(x) &= \scal{\mathcal{C}_{{\psi_2}} F, S^\nu(\cdot,x)}_{\Lnuc}
 = \int_{\C}F(\xi) S^\nu(\overline{\xi},x) e^{-\nu |\xi|^2} d\lambda(\xi),
 \end{align*}
where $S^\nu$ denotes the generating function of the rescaled Hermite polynomials $H^\nu_m$ given by
 \begin{align}
 S^\nu(q,x) &= \left(\frac{\nu}{\pi}\right)^{\frac 12}\sum_{m=0}^{+\infty}
\left(\frac{\nu^m}{m!}\right)^{\frac 12} \frac{q^m  H^\nu_m(x)}{\norm{H^\nu_m}_{\Lnur}} = \left(\frac{\nu}{\pi}\right)^{\frac 34} e^{-\frac{\nu}{2}q^2 +\sqrt{2}\nu qx}.\label{Genfct1}
 \end{align}
 Such kernel function satisfies
  \begin{eqnarray}\label{genss1}
 \scal{ S^\nu(q,\cdot), S^\nu(\xi,\cdot)}_{\Lnur}
 = \left(\frac{\nu}{\pi}\right) \sum_{m=0}^{+\infty} \frac{\nu^m q^m \overline{\xi}^m}{m!} \nonumber
  =: \left(\frac{\nu}{\pi}\right) e_{*}^{\nu[q,\overline{\xi}]} \nonumber
  = K^\nu_{\Hq}(q,\xi).
 \end{eqnarray}
Thus, for every $F\in \mathcal{F}^{2,\nu}(\C^2)$ and $q\in \C_i\simeq \C$, we have
    \begin{eqnarray}
 \mathcal{B}^\nu_\Hq \circ \mathcal{R}^\nu F  (q)
 &=& \scal{\mathcal{C}_{{\psi_2}} F, \scal{ S^\nu(q,\cdot), S^\nu(\cdot\cdot,\cdot) }_{\Lnur}  }_{\Lnuc} \nonumber
 \\ &=&  \left(\frac{\nu}{\pi}\right) \int_{\C} F\left(\frac{\xi}{\sqrt{2}}, \frac{-i\xi}{\sqrt{2}} \right) e_{*}^{\nu[q,\overline{\xi}]}  e^{-\nu |\xi|^2} d\lambda(\xi) \label{BR}
  \\&=&  \left(\frac{\nu}{\pi}\right) \int_{\C} F\left(\frac{\xi}{\sqrt{2}}, \frac{-i\xi}{\sqrt{2}} \right) e^{\nu q \overline{\xi}}  e^{-\nu |\xi|^2} d\lambda(\xi)  \nonumber
 \\&
=& F\left(\frac{q}{\sqrt{2}}, \frac{-iq}{\sqrt{2}} \right)  \nonumber
\\&
=:& \mathcal{I}^\nu F(q),  \nonumber
 \end{eqnarray}
  since $({\nu}/{\pi}) e^{\nu q \overline{\xi}}$ is the reproducing kernel of $\mathcal{F}^{2,\nu}(\C)$ and $\xi \longmapsto F\left(\frac{\xi}{\sqrt{2}}, \frac{-i\xi}{\sqrt{2}} \right) \in \mathcal{F}^{2,\nu}(\C)$.
The proof is completed.
\end{proof}

The following result identifies $\mathcal{I}^\nu(\mathcal{F}^{2,\nu}(\C ^2))$ as the specific subspace of slice regular functions in $\mathcal{F}_{slice}^{2,\nu}(\Hq)$ leaving the slice $\C_i$ invariant,
$$\mathcal{F}_{slice,i}^{2,\nu}(\Hq) :=\{ F\in \mathcal{F}_{slice}^{2,\nu}(\Hq); \, F(\C_i)\subset \C_i \}.$$
Its sequential characterization reads
$$\mathcal{F}_{slice,i}^{2,\nu}(\Hq) = \left\{ F(q) =\sum_{m=0}^{+\infty} q^{m}  c_m; \, c_m \in \C_i, \sum_{m=0}^{+\infty}\frac{m!}{\nu^m}|c_{m}|^{2}<+\infty \right\} . $$

\begin{thm} \label{I}
The transform $\mathcal{I}^\nu$ maps $\mathcal{F}^{2,\nu}(\C^2)$ onto $\mathcal{F}_{slice,i}^{2,\nu}(\Hq)$ and its action on the reproducing kernel $K^\nu_2((u,v),(z,w))$ in \eqref{Rpkernel2} is given by
\begin{eqnarray}\label{Ikernel}
\mathcal{I}^\nu(K^\nu_2(\cdot,(z,w)))(q) = K^\nu_{\Hq}\left(q,\frac{z+iw}{\sqrt{2}}\right).
\end{eqnarray}
\end{thm}

  \begin{proof} Let $F(z,w) =\sum\limits_{m,n=0}^{+\infty} a_{m,n} e_{m,n}(z,w)\in \mathcal{F}^{2,\nu}(\C^2)$, where $e_{m,n}(z,w)=z^m w^n$.
  By means of Theorem \ref{Iintclosed}, we have $ \mathcal{I}^\nu F = \mathcal{B}^\nu_\Hq \circ \mathcal{R}^\nu F \in \mathcal{F}_{slice}^{2,\nu}(\Hq)$. Moreover, for every $q\in \Hq$, we have
  $$ \mathcal{I}^\nu(e_{m,n})(q)= Ext (\mathcal{C}_{\psi_2} e_{m,n})(q) = q^{m+n} (-i)^n 2^{-\frac{m+n}{2}} $$
 since $\mathcal{C}_{\psi_2} e_{m,n}(\xi) = (-i)^n 2^{-\frac{m+n}{2}} \xi^{m+n}$. Therefore
    $$
    \mathcal{I}^\nu(f)(q) 
    = \sum\limits_{j=0}^{+\infty}  q^{j}\left(\sum\limits_{k=0}^{j} (-i)^k 2^{-\frac{j}{2}} a_{j-k,k} \right)
    = \sum\limits_{j=0}^{+\infty}  q^{j} b_{j} ,
    $$
  where the coefficients $b_{j}= \sum\limits_{k=0}^{j} (-i)^k 2^{-\frac{j}{2}} a_{j-k,k}$ belong to $\C_i$.
 This shows that $\mathcal{I}^\nu(\mathcal{F}^{2,\nu}(\C^2) ) \subset \mathcal{F}_{slice,i}^{2,\nu}(\Hq) $.
 For the inverse inclusion, let $F \in \mathcal{F}_{slice,i}^{2,\nu}(\Hq)$ and let $f\in\Lnurh$ such that $F=\mathcal{B}^\nu_\Hq f$.
 Now, since $F(\C_i)\subset \C_i$ we get $f_0\in\Lnur$ and therefore $f_0 = \mathcal{R}^\nu F_0$ for some $F_0\in \mathcal{F}^{2,\nu}(\C^2)$. Thus,
 $F = \mathcal{B}^\nu_\Hq \circ \mathcal{R}^\nu F_0 = \mathcal{I}^\nu F_0$.

    The formula \eqref{Ikernel} for arbitrary fixed $(z,w)\in \C^2$ immediately follows from the identity principle (Lemma \ref{IdentityPrinciple}) for the left slice regular functions. Indeed, the left slice regular functions
\begin{align}
 q \longmapsto  \mathcal{I}^\nu(K^\nu_2(\cdot,(z,w)))(q) = Ext\left( \xi \longmapsto K^\nu_2\left(\left(\frac{\xi}{\sqrt{2}},-\frac{i\xi}{\sqrt{2}}\right),(z,w) \right)\right)(q)
\end{align}
and
\begin{align}
 q \longmapsto  K^\nu_{\Hq}\left(q,\frac{z+iw}{\sqrt{2}}\right) = \left(\frac{\nu}{\pi}\right) e_{*}^{\nu\left[q,\frac{\overline{z+iw}}{\sqrt{2}}\right]}
\end{align}
coincides on the slice $\C_i$ and therefore their difference is identically zero on the whole $\Hq$.
\end{proof}

 \begin{rem} For $F(q) =\sum_{m=0}^{+\infty} q^{m}  c_m\in \mathcal{F}_{slice,i}^{2,\nu}(\Hq)$, i.e., with
 $c_m \in \C_i$ and $\sum_{m=0}^{+\infty}\frac{m!}{\nu^m}|c_{m}|^{2}<+\infty $, then the function $f_0 = \mathcal{R}^\nu F_0$ involved in the above proof is given by
$$ f_0(x) = \sum_{m=0}^{+\infty} \frac{\norm{e_m}_{L^{2,\nu}(\C,\C)}}{\norm{H^\nu_{m}}_{\Lnur}}  c_m H^\nu_{m}(x)\in \Lnur.$$
 Moreover, we have $\norm{f_0}_{\Lnur} = \norm{F}_{L^{2,\nu}(\C,\C)}$.
\end{rem}

The last result of this section concerns the following integral transform
$$\mathcal{J}^{\nu}:=\mathcal{G}^{\nu} \circ (\mathcal{B}^\nu_\Hq)^{-1}$$
 mapping $\mathcal{F}_{slice,i}^{2,\nu}(\Hq)$ into the two-dimensional Bargmann-Fock space  $\mathcal{F}^{2,\nu}(\C^2)$ and suggested by the commutative diagram
 \[
\xymatrix{
\mathcal{F}_{slice,i}^{2,\nu}(\Hq)\ar[r]^{\mathcal{J}^{\nu}} \ar[d]_{(\mathcal{B}^\nu_\Hq)^{-1}}& \mathcal{A}^{2,\nu}(\C^2)\\
\Lnur\ar[ru]_{\mathcal{G}^\nu} }.
\]

 \begin{thm} The image of $\mathcal{J}^{\nu}$ coincides with $ \mathcal{A}^{2,\nu}(\C^2)$ in \eqref{Image},  and its action on any $f\in\mathcal{F}_{slice,i}^{2,\nu}(\Hq)$ is given by
\begin{align}\label{actionFv}
 \mathcal{J}^{\nu} F(z,w)= \left(\frac{\nu}{\pi}\right)^{\frac 12} F\left(\frac{z+iw}{\sqrt{2}}\right).
 \end{align}
Moreover, for every fixed $\xi\in \C$, we have
\begin{align}\label{Ckernel}
\mathcal{J}^{\nu} \left( K^\nu_{\Hq}(\cdot,\xi) \right)(z,w) = K^\nu_2 \left(\left(\frac{\xi}{\sqrt{2}},\frac{-i\xi}{\sqrt{2}}\right),(z,w)\right)
\end{align}
where $K^\nu_{\Hq}(q,\xi)$ and $K^\nu_2 \left((u,v),(z,w)\right)$ are the reproducing kernel of $\mathcal{F}_{slice}^{2,\nu}(\Hq)$ and $\mathcal{F}^{2,\nu}(\C^2)$ given by \eqref{RpKernelH} and \eqref{Rpkernel2} respectively.
\end{thm}

\begin{proof} Below, we identify $\C$ and $\C_i$. The restriction of $(\mathcal{B}^\nu_\Hq)^{-1}$ to $\mathcal{F}_{slice,i}^{2,\nu}(\Hq)$ has as image $\Lnur$ which is contained in $\Lnurh$. This readily follows by proceeding in a similar way as in Theorem \ref{MainThm2} since the rescaled Hermite polynomials $H_m^\nu$ is an orthogonal basis of $\Lnur$. Thus, by Theorem \ref{MainThm2}, we obtain
$$ \mathcal{G}^{\nu} \circ (\mathcal{B}^\nu_\Hq)^{-1} ( \mathcal{F}_{slice,i}^{2,\nu}(\Hq)) =  \mathcal{G}^{\nu}(\Lnur) = \mathcal{A}^{2,\nu}(\C^2).$$
This can also be reproved since
\begin{align}\label{Cbasis}
\mathcal{J}^{\nu}(e_m)(z,w) = \left(\frac{z+iw}{\sqrt{2}}\right)^m = e_m(z,w)
\end{align}
which immediately follows from the formula \eqref{Ckernel},
whose the proof can be handled by direct computation. Indeed, for given $F\in \mathcal{F}_{slice,i}^{2,\nu}(\Hq)$, we have $F(\C_i)\subset \C_i$ and $(\mathcal{B}^\nu_\Hq)^{-1}F = (\mathcal{B}^{1,\nu})^{-1} F_i $,
where $(\mathcal{B}^{1,\nu})^{-1}$ is the inverse of the one-dimensional Segal-Bargmann transform and $F_i$ is the restriction of $F$ to the slice $\C_i$. Then, the proof is completed making use of the definition of $\mathcal{G}^\nu f(z,w) = \left(\frac{\nu}{\pi}\right)^{\frac{1}{2}} \mathcal{C}_{\psi_1} (\mathcal{B}^{1,\nu} f)(z,w)$.
\end{proof}

\begin{rem} The restriction of $\mathcal{I}^{\nu}= \mathcal{B}^\nu_\Hq \circ \mathcal{R}^{\nu}$ to $ \mathcal{A}^{2,\nu}(\C^2)$ is the inverse of $\mathcal{J}^{\nu}:=\mathcal{G}^{\nu} \circ (\mathcal{B}^\nu_\Hq)^{-1}$ for satisfying
$\mathcal{J}^{\nu} \circ \mathcal{I}^{\nu} = Id_{\mathcal{A}^{2,\nu}(\C^2)}.$ 
\end{rem}

 In the next section, we investigate further properties of the integral transform $\mathcal{G}^\nu$ when combined with the Fourier transform and  connecting one and two-dimensional Bargmann-Fock spaces. We also discuss possible generalization to $d$-complex space $\C^d$.

\section{Appendix}

  We consider the rescaled Fourier transform $\widetilde{\mathcal{F}}^\nu_{\mp} $ defined on $\Lnur$ by
  $\widetilde{\mathcal{F}}^\nu_{\mp}    =  \mathcal{M}_{\nu/2}  \mathcal{F}^\nu_{\mp}   \mathcal{M}_{-\nu/2} $, where $ \mathcal{M}_\alpha$ denotes the ground state transform $\mathcal{M}_\alpha f:=e^{-\alpha\vert{z}\vert^2}f$, and $\mathcal{F}^\nu $ is the standard Fourier transform on $L^{2,0}(\R,\C)=L^2(\R,dx)$ with
  $$
  \mathcal{F}^\nu_{\mp} (\varphi)(x):= \left(\frac{\nu  }{2\pi}\right)^{\frac 12} \int_{\R} \varphi(u)e^{{\mp} \nu ixu}dx.$$
  More explicitly, $\widetilde{\mathcal{F}}^\nu_{\mp}   $ acts on $\Lnur$ as a bounded linear operator by
   \begin{eqnarray}\label{Fourier}
   \widetilde{\mathcal{F}}^\nu_{\mp}  (\varphi) (x)
   := \left(\frac{\nu  }{2\pi}\right)^{\frac 12} \int_{\R} \varphi (u) e^{\frac{\nu}{2}(x {\mp}  iu)^2}   d\lambda(u).
  \end{eqnarray}
 Thanks to the well-known Plancherel-theorem, it turns out that the Fourier transform $\widetilde{\mathcal{F}}^\nu_{\mp}$ maps unitary $\Lnur$  onto itself.
 Accordingly, we can consider the following commutative diagrams
    $$\xymatrix{
    \mathcal{F}^{2,\nu}(\C) \ar[r]^{\mathcal{T}^\nu_{1,\mp}} \ar[d]_{(\mathcal{B}^{1,\nu})^{-1}} & \mathcal{A}^{2,\nu}(\C^2)  \\ \Lnur   \ar[r]_{\widetilde{\mathcal{F}}^\nu_{\mp}  } & \Lnur \ar[u]_{\mathcal{G}^\nu}
  }
  \quad \mbox{and} \quad \xymatrix{
    \mathcal{F}^{2,\nu}(\C^2) \ar[r]^{\mathcal{T}^\nu_{2,\mp}} \ar[d]_{\mathcal{R}^{\nu}} & \mathcal{F}^{2,\nu}(\C)  \\ \Lnur   \ar[r]_{\widetilde{\mathcal{F}}^\nu_{\mp}} & \Lnur \ar[u]_{\mathcal{B}^{1,\nu}} } .
  $$
  The transform
     $\mathcal{T}^\nu_{1,\mp}:=\mathcal{G}^\nu \circ \widetilde{\mathcal{F}}^\nu_{\mp}   \circ(\mathcal{B}^{1,\nu})^{-1}
     $
     (resp. $ \mathcal{T}^\nu_{2,\mp} :=\mathcal{B}^{1,\nu} \circ \widetilde{\mathcal{F}}^\nu _{\mp}   \circ \mathcal{R}^\nu  $) maps 
     $\mathcal{F}^{2,\nu}(\C)$ onto $\mathcal{A}^{2,\nu}(\C^2)$ (resp. $\mathcal{F}^{2,\nu}(\C^2)$ onto $\mathcal{F}^{2,\nu}(\C)$). Their explicit formulas reduced further to elementary composition operators involving the symbols $ {\psi_1}(z,w) = \frac{z+iw}{\sqrt{2}}$ and $ {\psi_2}(\xi) = \frac{1}{\sqrt{2}}(\xi,-i\xi)$, and the reducible representation of the unitary group $U(1):=\{\theta \in \C; \, |\theta |=1 \}$ defined by $\Gamma_\theta \varphi (\xi) := \varphi (\theta \xi)$.

 \begin{thm} \label{thmFourier}
  The action of $\mathcal{T}^\nu_{1,\mp}$ and $\mathcal{T}^\nu_{2,\mp}$ are given respectively by
\begin{align}\label{Fouc1}
\mathcal{T}^\nu_{1,\mp}  = \mathcal{B}^{2,\nu}|_{\mathcal{F}^{2,\nu}(\C )} \circ \Gamma_{{\mp} i} =\left(\frac{\nu}{\pi}\right)^{\frac{1}{2}}   \mathcal{C}_{{\mp}i\psi_1}
\end{align}
on $\mathcal{F}^{2,\nu}(\C )$, and
\begin{align}\label{Fouc2}
\mathcal{T}^\nu_{2,\mp}  = \Gamma_{{\mp}i} \circ Proj\circ (\mathcal{B}^{2,\nu})^{-1} = \left(\frac{\pi}{\nu}\right)^{\frac{1}{2}} \mathcal{C}_{{\mp}i{\psi_2}}
\end{align}
on $\mathcal{F}^{2,\nu}(\C^2)$. Moreover, we have $\mathcal{T}^\nu_{2,\mp}\circ \mathcal{T}^\nu_{1,\pm}= Id_{\mathcal{F}^{2,\nu}(\C )}$ and $\mathcal{T}^\nu_{2,\mp}\circ \mathcal{T}^\nu_{1,\mp}=\Gamma_{-1} Id_{\mathcal{F}^{2,\nu}(\C )}$.
  \end{thm}

    \begin{proof}
  Recall first that the expression of $(\mathcal{B}^{1,\nu})^{-1}$ given by
  $$(\mathcal{B}^{1,\nu})^{-1} f (x)= \scal{f,S^\nu(\cdot,x) }_{\Lnuc} ,$$
 where $S^\nu$ is the kernel function associated to the rescaled Hermite polynomials $H^\nu_m$ and given by \eqref{Genfct1}.
 Therefore, by Fubini's theorem, we get
 $$ \widetilde{\mathcal{F}}^\nu  \circ(\mathcal{B}^{1,\nu})^{-1} (f)(x) =
     \left(\frac{\nu  }{2\pi}\right)^{\frac 12} \int_{\C} f(\xi) \left( \int_\R e^{\frac{\nu}{2}(x {\mp} iu)^2} S^\nu(\overline{\xi},u) du\right) e^{-\nu|\xi|^2}d\lambda(u) .$$
 Straightforwardly, we obtain
 $$ \left(\frac{\nu  }{2\pi}\right)^{\frac 12}\int_\R e^{\frac{\nu}{2}(x {\mp} iu)^2} S^\nu(\zeta,u) du= S^\nu({\mp} i\zeta,x).$$
 Hence
\begin{align}\label{FouB}
\widetilde{\mathcal{F}}^\nu  \circ (\mathcal{B}^{1,\nu})^{-1} f (x)= \scal{\Gamma_{{\mp}i}f,S^\nu(\cdot,x) }_{\Lnuc} =  (\mathcal{B}^{1,\nu})^{-1} \circ \Gamma_{{\mp}i}f (x).
 \end{align}
 Consequently, the transform $\mathcal{T}^\nu_{1,\mp} = \mathcal{G}^\nu \circ  \widetilde{\mathcal{F}}^\nu_{\mp}  \circ(\mathcal{B}^{1,\nu})^{-1}$ reduces further to
 $$\mathcal{T}^\nu_{1,\mp} f(z,w) = \mathcal{B}^{2,\nu}\circ \mathcal{B}^{1,\nu} ( (\mathcal{B}^{1,\nu})^{-1} \circ \Gamma_{{\mp}i}f)(z,w)
 = \mathcal{B}^{2,\nu} f({\mp}iz,{\mp}iw)$$
 by means of Theorem \ref{MainThm}, as well as to
\begin{align*}
\mathcal{T}^\nu_{1,\mp}  
 = \left(\frac{\nu}{\pi}\right)^{\frac{1}{2}} \mathcal{C}_{\psi_1} \circ \mathcal{B}^{1,\nu}  (\mathcal{B}^{1,\nu})^{-1} \circ \Gamma_{{\mp}i}
 = \left(\frac{\nu}{\pi}\right)^{\frac{1}{2}}  \Gamma_{{\mp}i}\circ \mathcal{C}_{\psi_1}
 =  \left(\frac{\nu}{\pi}\right)^{\frac{1}{2}}   \mathcal{C}_{{\mp}i\psi_1}
\end{align*}
on $\mathcal{F}^{2,\nu}(\C)$.
Moreover, by Theorem \ref{thmLinverse} and \eqref{FouB}, the action of $\mathcal{T}^\nu_{2,\mp} :=\mathcal{B}^{1,\nu} \circ \widetilde{\mathcal{F}}^\nu  \circ \mathcal{R}^\nu$ on $\mathcal{F}^{2,\nu}(\C^2)$ reads
\begin{align*}
\mathcal{T}^\nu_{2,\mp} 
=  \left(\frac{\pi}{\nu}\right)^{\frac{1}{2}} \mathcal{B}^{1,\nu} \circ \widetilde{\mathcal{F}}^\nu_{\mp}  \circ (\mathcal{B}^{1,\nu})^{-1} \mathcal{C}_{{\psi_2}}
 =\left(\frac{\pi}{\nu}\right)^{\frac{1}{2}} \Gamma_{{\mp}i} \circ \mathcal{C}_{{\psi_2}}
 =\left(\frac{\pi}{\nu}\right)^{\frac{1}{2}} \mathcal{C}_{{\mp}i{\psi_2}} .
\end{align*}
We also have
\begin{align*}
\mathcal{T}^\nu_{2,\mp}  &:= \mathcal{B}^{1,\nu} \circ \widetilde{\mathcal{F}}^\nu_{\mp}  \circ \mathcal{R}^\nu
\\& =\mathcal{B}^{1,\nu} \circ \widetilde{\mathcal{F}}^\nu_{\mp} \circ (\mathcal{B}^{1,\nu})^{-1}\circ Proj\circ (\mathcal{B}^{2,\nu})^{-1}
\\& = \Gamma_{{\mp}i} \circ Proj\circ (\mathcal{B}^{2,\nu})^{-1} .
\end{align*}
Finally, from \eqref{Fouc1} and \eqref{Fouc2}, we obtain
\begin{align*}
\mathcal{T}^\nu_{2,\mp} (\mathcal{T}^\nu_{1,\mp} f)(\xi) &= \mathcal{C}_{{\mp}i{\psi_2}} (\mathcal{C}_{{\mp}i{\psi_1}} f)(\xi)
\\&=  \mathcal{C}_{{\mp}i{\psi_1}} f \left(\frac{\mp i \xi}{\sqrt 2}, \frac{{\mp}\xi}{\sqrt 2}\right)
\\&= f(-\xi)
\end{align*}
as well as
\begin{align*}
\mathcal{T}^\nu_{2,\mp} (\mathcal{T}^\nu_{1,\pm} f)(\xi) & = \mathcal{C}_{{\mp}i{\psi_2}} (\mathcal{C}_{{\pm}i{\psi_1}} f)(\xi)
\\&=  \mathcal{C}_{{\pm}i{\psi_1}} f \left(\frac{\mp i \xi}{\sqrt 2}, \frac{{\mp}\xi}{\sqrt 2}\right)
\\&= f(\xi).
\end{align*}
  \end{proof}

We conclude this paper by discussing the generalization to $d$-complex space $\C^d$. This is possible for $d=2^k$ by considering the integral transform
$\mathcal{G}^\nu_k$ mapping isometrically the standard Hilbert space $\Lnur)$ into the Bargmann-Fock space $\mathcal{F}^{2,\nu}(\C^{2^k})$ defined by induction
$$\mathcal{G}^\nu_k:=
 \mathcal{B}^{2^{k},\nu} \circ \mathcal{B}^{2^{k-1},\nu} \circ \cdots \circ  \mathcal{B}^{2,\nu} \circ \mathcal{B}^{1,\nu}. $$
We claim that for every $f\in \Lnur$ and $Z=(z_1, \cdots ,z_{2^k})\in \C ^{2^k}$ we have
$$\mathcal{G}^\nu_k f(Z)=  c_{2^{k}}^\nu \mathcal{C}_{{\psi_k}}\mathcal{B}^{1,\nu} f(Z)= c_{2^{k}}^\nu \mathcal{B}^{1,\nu} f({\psi_k}(Z))$$
 where $\mathcal{C}_{{\psi_k}}$ denotes the composition operator with the special symbol
 $${\psi_k}(Z):=  \frac{1}{2^{\frac{k}{2}}}\sum_{m=0}^{2^{k-1}-1}i^m(z_{2m+1}+iz_{2m+2}).$$
The computations hold true for $k=1$ and $k=2$.

%
%
%


\begin{thebibliography}{22}
\bibitem{AlpayColomboSabadini2014} Alpay D., Colombo F., Sabadini I., Salomon G.,
       { The Fock space in the slice hyperholomorphic Setting}.
       In Hypercomplex Analysis: New perspectives and applications.
       Trends Math. (2014) 43--59.
\bibitem{Bargmann1961} Bargman V.,
        { On a Hilbert space of analytic functions and an associated integral transform}.
        Comm. Pure Appl. Math. 14 (1961) 187-214.
\bibitem{CD2017} Cnudde L., De Bie B.,
         { Slice Segal-Bargmann transform}.  C Journal of Physics A: Mathematical and Theoretical, Volume 50, Number 25 (2017)
\bibitem{ColomboSabadiniStruppa2011} Colombo F., Sabadini I., Struppa D.C.,
       { Noncommutative Functional Calculus, Theory and Applications
         of Slice Hyperholomorphic Functions}.
         Progress in Mathematics, 289. Birkh\"auser, Basel (2011).
\bibitem{DMNQ2016} Dang P., Mourao, J. Nunes J.P.,  Qian T.,
       { Clifford coherent state transforms on spheres}. 	
   J. Geom. Phys. 124 (2018), 225--232.
\bibitem{DG1.2017} Diki K., Ghanmi A.,
         { A quaternionic analogue of the Segal-Bargmann transform}.
         Complex Anal. Oper. Theory, 11, 457-473 (2017)
\bibitem{Folland1989} Folland G.B.,
       { Harmonic analysis in phase space}.
            Annals of Mathematics Studies, 122.
            Princeton University Press, Princeton, NJ; 1989.
\bibitem{GS2007} Gentili G., Struppa D.C.,
        { A new theory of regular functions of a quaternionic variable}.
        Adv. Math. 216, 279-301 (2007)
\bibitem{GentiliStoppatoStruppa2013} Gentili G., Stoppato C., Struppa D.C.,
        {  Regular functions of a quaternionic variable}.
        Springer Monographs in Mathematics, 2013.
\bibitem{Hall2013} Hall B.C., 
       {  Quantum Theory for Mathematicians}.
       Graduate texts in Mathematics. 2013.
   \bibitem{KMNQ2016} Kirwin, W.D., Mourao, J., Nunes, J. P.,  Qian, T. {\it Extending coherent state transforms to Clifford analysis}. J. Math. Phys. 57, 103505 (2016)
 \bibitem{MouraoNQ2017} Mourao, J., Nunes, J. P.,  Qian, T.
           { Coherent State Transforms and the Weyl Equation in Clifford Analysis}.
          J. Math. Phys. 58 (2017), no. 1, 013503, 12
\bibitem{Neretin1972} Neretin Y.A.,
         {  Lectures on gaussian integral operators and classical groups}.
          EMS Series of Lectures in Mathematics. European Mathematical Society (EMS), Z\"urich, 2011
\bibitem{PSS2016} Pena, D.P., Sabadini, I., Sommen, F.:{\it Segal-Bargmann-Fock modules of monogenic functions}.
J. Math. Phys. 58 (2017), no. 10, 103507
\bibitem{Zhu2012}  Zhu K.,
        { Analysis on Fock spaces}.
        Graduate Texts in Mathematics, 263. Springer, New York, 2012.
\end{thebibliography}
\end{document}